\documentclass[oneside,english]{amsart}
\usepackage[T1]{fontenc}
\usepackage[latin9]{inputenc}
\usepackage{amsthm}
\usepackage{amstext}
\usepackage{amssymb}
\usepackage{esint}

\makeatletter
\numberwithin{equation}{section}
\numberwithin{figure}{section}
\theoremstyle{plain}
\newtheorem{thm}{\protect\theoremname}[section]

\theoremstyle{plain}
\theoremstyle{definition}

\theoremstyle{plain}

\theoremstyle{plain}
\newtheorem{rem}[thm]{\protect\remarkname}
\theoremstyle{plain}
\makeatother

\usepackage{babel}
\providecommand{\definitionname}{Definition}
\providecommand{\lemmaname}{Lemma}
\providecommand{\theoremname}{Theorem}
\providecommand{\corollaryname}{Corollary}
\providecommand{\remarkname}{Remark}
\providecommand{\propositionname}{Proposition}
	
\DeclareMathOperator{\loc}{loc}

\DeclareMathOperator{\ess}{ess}

\DeclareMathOperator{\cp}{cap}

\begin{document}

\title{Composition Operators on Sobolev Spaces and Neumann Eigenvalues}

\author{V.~Gol'dshtein, A.~Ukhlov}
\begin{abstract}
In this paper we discuss applications of the geometric theory of composition operators on Sobolev spaces to the spectral theory of non-linear elliptic operators. The lower  estimates of the first non-trivial Neumann eigenvalues of the $p$-Laplace operator in cusp domains $\Omega\subset\mathbb R^n$, $n\geq 2$, are given.

\end{abstract}
\maketitle
\footnotetext{\textbf{Key words and phrases:} Sobolev spaces, Neumann eigenvalues, Quasiconformal mappings} 
\footnotetext{\textbf{2000
Mathematics Subject Classification:} 35P15, 46E35, 30C65.}

\section{Introduction }

In the article we discuss a general method of applications of the geometric theory of composition operators on Sobolev spaces to spectral estimates of the non-linear Neumann-Laplace operators, proposed in our recent works.
Suppose that in a bounded domain  $\Omega\subset\mathbb R^n$  the following Sobolev-Poincar\'e inequality 
$$
\inf\limits_{c\in\mathbb R}\|f-c \mid L_p(\Omega)\|\leq B_{p,p}(\Omega)\|\nabla f \mid L_p(\Omega)\|,\,\,\, f\in W^{1}_{p}(\Omega).
$$
holds with the Poincar\'e constant $B_{p,p}(\Omega)$.
Then the first nontrivial Neumann eigenvalue $\mu_{p}(\Omega)$ of the $p$-Laplace operator 
\begin{equation}
\label{pLaplace}
-\Delta_p u=-\operatorname{div}\left(|\nabla u|^{p-2}\nabla u\right),\,\,p>1,
\end{equation}
can be characterized as $\mu_{p}(\Omega)=B^{-p}_{p,p}(\Omega)$ (see, for example, \cite{BCT15}). 

Exact calculations of Neumann eigenvalues are known in very limited number of domains and so estimates of $\mu_{p}(\Omega)$ represent an important part of the modern spectral theory (see, for example, \cite{A98,BCT15,FNT,ENT,EP15,LM98}).  We suggest spectral estimates of the first nontrivial Neumann eigenvalue $\mu_p(\Omega)$ of the $p$-Laplace operator in terms of the (weak) quasiconformal geometry of a domain $\Omega$.
 
The classical upper estimate for the first nontrivial Neumann eigenvalue of the Laplace operator   
$$
\mu_2(\Omega)\leq \mu_2(\Omega^{\ast})=\frac{p^2_{n/2}}{R^2_{\ast}}
$$
was proved by Szeg\"o \cite{S54} for simply connected planar domains via a conformal mappings technique ("the method of conformal normalization") and by Weinberger  \cite{W56} for  domains in $\mathbb{R}^n$. In this inequality $p_{n/2}$ denotes the first positive zero of the function $(t^{1-n/2}J_{n/2}(t))'$, and $\Omega^{\ast}$ is an $n$-ball of the same $n$-volume as $\Omega$ with $R_{\ast}$ as its radius. In particular, if $n=2$, we have $p_1=j_{1,1}'\approx1.84118$ where $j_{1,1}'$ denotes the first positive zero of the derivative of the Bessel function $J_1$. 

Is convex domains $\Omega\subset\mathbb R^n$, $n\geq 2$,  the classical lower estimates of the linear Neumann-Laplace operator (p=2) \cite{PW} states that   
$$
\mu_2(\Omega)\geq \frac{\pi^2}{d(\Omega)^2},
$$
where $d(\Omega)$ is a diameter of a convex domain $\Omega$. Similar estimates for $p\ne 2$ were obtained much later \cite{ENT}: 
$$
\mu_p(\Omega) \geq \left(\frac{\pi_p}{d(\Omega)}\right)^p
$$
where
$$
\pi_p = 2 \int\limits_0^{(p-1)^{\frac{1}{p}}} \frac{dt}{(1-t^p/(p-1))^{\frac{1}{p}}} =
2 \pi \frac{(p-1)^{\frac{1}{p}}}{p \sin(\pi/p)}.
$$

Unfortunately in non-convex domains $\mu_p(\Omega)$ can not be estimated in the terms of Euclidean diameters. It can be seen by considering a domain consisting of two identical squares connected by a thin corridor \cite{BCDL16}. 
In our previous works  \cite{GU16,GU2016} we returned to a conformal mappings techniques (that was used in \cite{S54}) in a framework of composition operators on Sobolev spaces. It permit us to obtain lower estimates of $\mu_2(\Omega)$  in the terms of the hyperbolic (conformal) radius of $\Omega$ for a large class of general (non necessary convex) domains $\Omega\subset\mathbb R^2$. 

In the case of space domains conformal mappings does not allow to obtain spectral estimates of the Neumann-Laplace operator and we use a generalization of conformal mappings such as weak $p$-quasiconformal mappings \cite{GU17}. The aim of the present work is to obtain lower estimates of the Neumann eigenvalues of the $p$-Laplace operator in the terms of general weak $(p,q)$-quasiconformal geometry of domains. We use one more time the geometric theory of composition operators on Sobolev spaces \cite{GGR95,U93} with applications to the (weighted) Poincar\'e-Sobolev inequalities \cite{GGu,GU}.

The paper is organized as follows:
The capacitory description of Sobolev spaces and a short survey of the geometric theory of composition operators on Sobolev spaces are presented in section 2. In section 3 we give a general method of applications of composition operators to the Sobolev-Poincar\'e inequalities and apply them to lower estimates of the first nontrivial Neumann eigenvalue for the $p$-Laplace operator.

\section{Composition operators on Sobolev spaces}

\subsection{Sobolev spaces}

Let $E$ be a measurable subset of $\mathbb R^n$, $n\geq 2$. The Lebesgue space $L_p(E)$, $1\leq p\leq\infty$, is defined as a Banach space of $p$-summable functions $f:E\to \mathbb R$ equipped with the following norm:
$$
\|f\mid L_p(E)\|=
\biggr(\int\limits_E|f(x)|^p\,dx\biggr)^{\frac{1}{p}},\,\,\,1\leq p<\infty,
$$
and
$$
\|f\mid L_{\infty}(E)\|=
\ess\sup\limits_{x\in E}|f(x)|,\quad p=\infty.
$$

If $\Omega$ is an open subset of $\mathbb R^n$, the Sobolev space $W^1_p(\Omega)$, $1\leq p\leq\infty$, is defined 
as a Banach space of locally integrable weakly differentiable functions
$f:\Omega\to\mathbb{R}$ equipped with the following norm: 
\[
\|f\mid W^1_p(\Omega)\|=\| f\mid L_p(\Omega)\|+\|\nabla f\mid L_p(\Omega)\|,
\]
where $\nabla f$ is the weak gradient of the function $f$, i.~e. $ \nabla f = (\frac{\partial f}{\partial x_1},...,\frac{\partial f}{\partial x_n})$,
$$
\int\limits_{\Omega} f \frac{\partial g}{\partial x_i}~dx=-\int\limits_{\Omega} \frac{\partial f}{\partial x_i} g~dx, \quad \forall g\in C_0^{\infty}(\Omega),\quad i=1,...,n.
$$ 
As usual, $C_0^{\infty}(\Omega)$ is the space of infinitely smooth functions with a compact support.

The homogeneous seminormed Sobolev space $L^1_p(\Omega)$, $1\leq p\leq\infty$, is defined as a space
of locally integrable weakly differentiable functions $f:\Omega\to\mathbb{R}$ equipped
with the following seminorm: 
\[
\|f\mid L^1_p(\Omega)\|=\|\nabla f\mid L_p(\Omega)\|.
\]

Recall the notion of the $p$-capacity of a set $E\subset \Omega$. Let $\Omega$ be a domain in $\mathbb R^n$ and a compact $F\subset\Omega$. The $p$-capacity of the compact $F$ is defined by
$$
\cp_p(F;\Omega) =\inf\{\|f|L^1_p(\Omega)|^p,\,\,f\geq 1\,\, \text{on}\,\, F, \,\,f\in C_0(\Omega)\cap L^1_p(\Omega),
$$
where $ C_0(\Omega)$ is the space of continuous functions with a compact support.

By the similar way we can define $p$-capacity of open sets.

For arbitrary set $E\subset\Omega$ we define an inner $p$-capacity as 
$$
\underline{\cp}_p(E;\Omega)=\sup\{\cp_p(e;\Omega),\,\,e\subset E\subset\Omega,\,\, e\,\,\text{is a compact}\},
$$
and an outer $p$-capacity as 
$$
\overline{\cp}_p(E;\Omega)=\inf\{\cp_p(U;\Omega),\,\,E\subset U\subset\Omega,\,\, U\,\,\text{is an open set}\}.
$$
A set $E\subset\Omega$ is called $p$-capacity measurable, if $\underline{\cp}_p(E;\Omega)=\overline{\cp}_p(E;\Omega)$. The value
$$
\cp_p(E;\Omega)=\underline{\cp}_p(E;\Omega)=\overline{\cp}_p(E;\Omega)
$$
is called the $p$-capacity of the set $E\subset\Omega$.

The notion of $p$-capacity permits us to refine the notion of Sobolev functions. Let a function $f\in L^1_p(\Omega)$. Then refined function 
$$
\tilde{f}(x)=\lim\limits_{r\to 0}\frac{1}{|B(x,r)|}\int\limits_{B(x,r)} f(y)~dy
$$
is defined quasieverywhere i.~e. up to a set of $p$-capacity zero and it is absolutely continuous on almost all lines \cite{M}. This refined function $\tilde{f}\in L^1_p(\Omega)$ is called a unique quasicontinuous representation ({\it a canonical representation}) of function $f\in L^1_p(\Omega)$. Recall that a function $\tilde{f}$ is termed quasicontinuous if for any $\varepsilon >0$ there is an open  set $U_{\varepsilon}$ such that the $p$-capacity of $U_{\varepsilon}$ is less than $\varepsilon$ and on the set $\Omega\setminus U_{\varepsilon}$ the function  $\tilde{f}$ is continuous (see, for example \cite{HKM,M}). 
In what follows we will use the quasicontinuous (refined) functions only.

Note that the first weak derivatives of
the function $f$ coincide almost everywhere with the usual
partial derivatives (see, e.g., \cite{M} ).

\subsection{Composition operators}

Let $\Omega$ and $\widetilde{\Omega}$ be domains in $\mathbb R^n$, $n\geq 2$. We say that
a homeomorphism $\varphi:\Omega\to\widetilde{\Omega}$ induces a bounded composition
operator 
\[
\varphi^{\ast}:L^1_p(\widetilde{\Omega})\to L^1_q(\Omega),\,\,\,1\leq q\leq p\leq\infty,
\]
by the composition rule $\varphi^{\ast}(f)=f\circ\varphi$, if for
any function $f\in L^1_p(\widetilde{\Omega})$, the composition $\varphi^{\ast}(f)\in L^1_q(\Omega)$
is defined quasi-everywhere in $\Omega$ and there exists a constant $K_{p,q}(\Omega)<\infty$ such that 
\[
\|\varphi^{\ast}(f)\mid L^1_q(\Omega)\|\leq K_{p,q}(\Omega)\|f\mid L^1_p(\widetilde{\Omega})\|.
\]

Composition operators on Sobolev spaces arise in \cite{M69} as operators generated by the class of sub-areal mappings. In \cite{VG75} in the frameworks of Reshetnyak's problem (1968) was proved that a homeomorphism $\varphi:\Omega\to\widetilde{\Omega}$ generates an isomorphism of Sobolev spaces $L^1_n(\Omega)$ and $L^1_n(\widetilde{\Omega})$ if and only if $\varphi$ is a quasiconformal mapping. These works motivated the study of geometric properties of mappings which generate isomorphisms (bounded operators) of Sobolev type spaces. In \cite{VG76}
was proved that a homeomorphism $\varphi:\Omega\to\widetilde{\Omega}$ generates an isomorphism of Sobolev spaces $L^1_p(\Omega)$ and $L^1_p(\widetilde{\Omega})$, $p>n$, if and only if $\varphi$ is a bi-Lipschitz mapping. This result was extended to the cases $n-1<p<n$ in  \cite{GR84} and the case $1\leq p<n$ in \cite{Mar90}.

\begin{rem} 
Recall that a homeomorphism $\varphi$ is called a bi-Lipschitz homeomorphism if $\varphi$ and $\varphi^{-1}$ are locally Lipschitz mappings with uniformly bounded local Lipschitz constants.
\end{rem}
 Mappings generates isomorphisms of Besov spaces was considered in \cite{V89}, Nikolskii-Besov spaces and Lizorkin-Triebel spaces were considered in \cite{V90}. In \cite{MS86} the theory of multipliers was applied to the change of variable problem in Sobolev spaces.

The composition operators are not necessary isomorphisms even if they induced by diffeomorphisms or homeomorphisms of Euclidean domains. It means that the composition problem can be reformulated by a more flexible way: describe classes of homeomorphisms that induce bounded composition operators on Sobolev spaces. In the case of Sobolev spaces with the first weak derivatives it can be formulated as a characterization of homeomorphism $\varphi:\Omega\to\widetilde{\Omega}$ that generate composition operators 
\begin{equation}
\label{eq:comppq}
\varphi^{\ast}: L^1_p(\widetilde{\Omega})\to L^1_q(\Omega),\,\,\,1\leq q\leq p\leq\infty,
\end{equation}
generated by the standard composition rule $\varphi^{\ast}(f)=f\circ\varphi$.

Analytical characteristics of mappings generate bounded compositions operators on Sobolev spaces are given in terms of weak derivatives of the mappings.
Let a mapping $\varphi:\Omega\to\mathbb{R}^{n}$ be weakly differentiable in $\Omega$. Then the formal Jacobi
matrix $D\varphi(x)$ and its determinant (Jacobian) $J(x,\varphi)$
are well defined at almost all points $x\in\Omega$. The norm $|D\varphi(x)|$
of the matrix $D\varphi(x)$ is the norm of the corresponding linear
operator. We will use the same notation for this matrix and the corresponding
linear operator.

Recall that a mapping $\varphi:\Omega\to\mathbb{R}^{n}$ possesses the Luzin $N$-property if an image of any set of measure zero has measure zero. Lipschitz mapping possess the Luzin $N$-property.

In the case $p=q\ne n$ an analytic description was obtained in \cite{V88} using a notion of mappings of finite distortion introduced in \cite{VGR}: a weakly differentiable mapping is called a mapping of finite distortion if $|D\varphi(x)|=0$ a.~e. on the set $Z=\{x\in\Omega: J(x,\varphi)=0\}$. This property becomes useful especially for the case $p\neq q$. 

In \cite{GGR95} it was obtained the geometric description of composition operators for $n-1<p=q<\infty$ without using the finite distortion property. In the case $1<p<n-1$ there are proofs \cite{GGR95} with additional apriori properties of homeomorphisms. 

Let us reformulate the main result of \cite{GGR95}. 
\begin{thm}
\label{CompGeom} \cite{GGR95} A homeomorphism $\varphi:\Omega\to\widetilde{\Omega}$
between two domains $\Omega$ and $\widetilde{\Omega}$ induces a bounded composition
operator 
\[
\varphi^{\ast}:L^1_p(\widetilde{\Omega})\to L^1_p(\Omega),\,\,\,n-1<p<\infty,
\]
if and only if 
\begin{equation}
\label{eq:geom}
M^{(\lambda)}_{p}(\varphi;\Omega)=\sup\limits_{x\in\Omega}\limsup\limits_{r\to 0}\frac{L^p_{\varphi}(x,r)r^{n-p}}{|\varphi(B(x,\lambda r))|}<\infty,
\end{equation} 
for some $\lambda>1$ if $n-1<p<n$ and $\lambda=1$ if $n\leq p<\infty$, where $L_{\varphi}(x,r)=\max\limits_{|x-y|=r}|\varphi(x)-\varphi(y)|$.

\end{thm}

On the base of this result the notion of $p$-quasiconformal mappings was introduced for $n-1<p<\infty$. In the case $1<p<n-1$ this result is correct under an additional apriori property of weak differentiability of $\varphi$ \cite{GGR95}. 

Recall that a weakly differentiable homeomorphism $\varphi:\Omega\to\widetilde{\Omega}$ is called a weak $p$-quasiconformal mapping if there exists a constant $K_p<\infty$ such that
$$
|D\varphi(x)|^p\leq K^p_p |J(x,\varphi|\,\,\,\text{for almost all}\,\,\, x\in\Omega.
$$

The class of weak $p$-quasiconformal mappings are natural generalization of quasiconformal mappings and for $p=n$ these classes coincide. 

Note that planar conformal mapping preserve the Dirichlet energy integral 
$$
\int\limits_{\Omega}|\nabla f(x)|^2~dx
$$ 
because $|D\varphi(x)|^2= |J(x,\varphi|$ in $\Omega$.

In the space $\mathbb R^3$ conformal mappings don't connected to the Dirichlet energy integral, but the class of weak $2$-conformal mappings  (co-conformal mappings \cite{GU17}) generates bounded composition operators on the Dirichlet spaces.  It is a reason for detailed study of weak $2$-quasiconformal mappings that looks as an appropriate class for possible applications to  elliptic equations as was mentioned in \cite{GU17}. 

The case $p\ne q$ is more complicated and in this case the composition operators theory is based on the countable-additive property  of the norm of composition operators defined on open subsets of $\Omega$ \cite{U93} (see also \cite{VU02}).
The main result of \cite{U93} gives an analytic description of composition operators on Sobolev spaces (see, also \cite{VU02}) and asserts that

\begin{thm}
\label{CompTh} \cite{U93} A homeomorphism $\varphi:\Omega\to\widetilde{\Omega}$
between two domains $\Omega$ and $\widetilde{\Omega}$ induces a bounded composition
operator 
\[
\varphi^{\ast}:L^1_p(\widetilde{\Omega})\to L^1_q(\Omega),\,\,\,1\leq q<p<\infty,
\]
 if and only if $\varphi\in W^1_{1,\loc}(\Omega)$, has finite distortion,
and 
\[
K_{p,q}(\varphi;\Omega)=\biggl(\int\limits _{\Omega}\biggl(\frac{|D\varphi(x)|^{p}}{|J(x,\varphi)|}\biggr)^{\frac{q}{p-q}}~dx\biggr)^{\frac{p-q}{pq}}<\infty.
\]
\end{thm}

{\it Let us remark one more time that in the framework of two previous theorems function of spaces $L^1_p$ are quasicontinuous.} If the concept of quasicontinuity is not used it can lead to wrong conclusions. In \cite{K12} by an analogy with Lebesgue spaces, Sobolev functions were defined up to set s measure zero (not up to sets of $p$-capacity zero).  Recall that by the Sobolev embedding theorem (see, for example, \cite{HKM,M}) any function $f\in W^1_p(\widetilde{\Omega})$, $p>n$, has an unique continuous representation \cite{M} and it is natural to consider only these continuous functions (in this case points have nonzero $p$-capacity, $p>n$).

Because weak $p$-quasiconformal mappings and weak $(p,q)$-quasiconformal mappings not necessary have the Luzin $N^{-1}$-property, i.e. pre-images of sets of a measure zero not necessary have a measure zero, this observation was used in \cite{K12} for the wrong conclusion that Theorem~\ref{CompGeom} and Theorem~\ref{CompTh} are not correct (especially in its sufficiency parts). We put here this detailed explanation because the mistake in  \cite{K12} was not fixed up to this moment. 

By an analogy with a notion of weak $p$-quasiconformal case homeomorphisms which satisfy conditions of Theorem~\ref{CompTh} are called (weak) $(p,q)$-quasiconformal mappings \cite{VU98} or mappings of bounded $(p,q)$-distortion \cite{UV10}. The geometric description similar (\ref{eq:geom}) of weak $(p,q)$-quasiconformal mappings in the case $q>n-1$ was obtained in \cite{VU98}.

Composition operators on Sobolev spaces have applications to spectral problems of elliptic equations. These applications are based on the Sobolev
type embedding theorems \cite{GGu,GU}. The following diagram illustrate the main idea of these applications:

\[\begin{array}{rcl}
W^{1}_{p}(\widetilde{\Omega}) & \stackrel{\varphi^*}{\longrightarrow} & W^1_q(\Omega) \\[2mm]
\multicolumn{1}{c}{\downarrow} & & \multicolumn{1}{c}{\downarrow} \\[1mm]
L_s(\widetilde{\Omega}) & \stackrel{(\varphi^{-1})^*}{\longleftarrow} & L_r(\Omega)
\end{array}\]

Here the operator $\varphi^{\ast}$ defined by the composition rule $\varphi^{\ast}(f)=f\circ\varphi$ is a bounded composition operator on Sobolev spaces induced by a homeomorphism $\varphi$ of a "good" domain $\Omega$ (for example the unit ball) to a "bad" domain  $\widetilde{\Omega}$ and the operator $(\varphi^{-1})^{\ast}$ defined by the composition rule $(\varphi^{-1})^{\ast}(f)=f\circ\varphi^{-1}$ is a bounded composition operator on Lebesgue spaces induced by the inverse homeomorphism. 

This approach allows us to obtain estimates of stability of constants in Sobolev-Poincar\'e inequalities in domains, that in its turn give us opportunity to estimates stability of Neumann eigenvalues of the $p$-Laplace operators.

In our recent works \cite{BGU1,BGU2,GPU17,GU16,GU2016} the spectral stability problem and the lower estimates of Neumann eigenvalues in planar domains were studied. In space domains our results are more modest. Some spectral estimates in space non convex domains were obtained \cite{GU17} with the help of weak $p$-quasiconformal mappings theory.

The geometric theory of composition operators on Sobolev spaces is closely connected with the $A_{p,q}$-Ball's classes \cite{B76,B81} and has applications in non-linear elasticity (see, for example, \cite{CHM,GU10,V00}).

In the last decade the composition operators theory on generalizations of Sobolev spaces, like Besov spaces and Triebel-Lizorkin spaces, \cite{HK13,KKSS14,KYZ11,KXZZ17,OP17} was under consideration. These types of composition operators have applications to the Calder\'on inverse conductivity problem \cite{C80}.

\subsection{Composition operators and capacity inequalities}
Composition operators on Sobolev spaces allow an alternative capacitory description. Recall the notion of a variational $p$-capacity \cite{GResh}.

A condenser in the domain $\Omega\subset \mathbb R^n$ is the pair $(F_0,F_1)$ of connected closed relatively to $\Omega$ sets $F_0,F_1\subset \Omega$. A continuous function $f\in L_p^1(\Omega)$ is called an admissible function for the condenser $(F_0,F_1)$,
if the set $F_i\cap \Omega$ is contained in some connected component of the set $\operatorname{Int}\{x\vert u(x)=i\}$,\ $i=0,1$. We call as the $p$-capacity of the condenser $(F_0,F_1)$ relatively to domain $\Omega$
the following quantity:
$$
{{\cp}}_p(F_0,F_1;\Omega)=\inf\|f\vert L_p^1(\Omega)\|^p.
$$
Here the greatest lower bond is taken over all functions admissible for the condenser $(F_0,F_1)\subset\Omega$. If the condenser has no admissible functions we put the capacity equal to infinity. 

The following theorems give the capacitory description of the composition operators on Sobolev spaces. 

\begin{thm}\cite{GGR95}
\label{thm:CapacityDescPP}
Let $1<p<\infty$.
A homeomorphism $\varphi :\Omega\to \widetilde{\Omega}$
generates a bounded composition operator
$$
\varphi^{\ast}: L^1_p(\widetilde{\Omega})\to L^1_p(\Omega)
$$
if and only if for every condenser 
$(F_0,F_1)\subset\widetilde{\Omega}$
the inequality
$$
\cp_{p}^{1/p}(\varphi^{-1}(F_0),\varphi^{-1}(F_1);\Omega)
\leq K\cp_{p}^{1/p}(F_0,F_1;\widetilde{\Omega})
$$
holds. 
\end{thm}

\begin{thm}\cite{VU98}
\label{thm:CapacityDescPQ}
Let $1<q<p<\infty$.
A homeomorphism $\varphi :\Omega\to \widetilde{\Omega}$
generates a bounded composition operator
$$
\varphi^{\ast}: L^1_p(\widetilde{\Omega})\to L^1_q(\Omega)
$$
if and only if
there exists a bounded monotone countable-additive set function
$\Phi$ defined on open subsets of $\widetilde{\Omega}$
such that for every condenser 
$(F_0,F_1)\subset \widetilde{\Omega}$
the inequality
$$
\cp_{q}^{1/q}(\varphi^{-1}(F_0),\varphi^{-1}(F_1);\Omega)
\leq\Phi(\widetilde{\Omega}\setminus(F_0\cup F_1))^{\frac{p-q}{pq}}
\cp_{p}^{1/p}(F_0,F_1;\widetilde{\Omega})
$$
holds. 
\end{thm}

This capacity inequalities demonstrate a close connection of  mappings that generate bounded composition operators on Sobolev spaces and so-called
$Q$-ho\-me\-o\-mor\-phisms \cite{MRSY}. Descriptions of  $Q$-homeomorphisms are based on the capacitory (moduli) distortion property are these classes were intensively studied at last decades (see, for example, \cite{K86,Sa15,S16}).

\section{Spectral estimates of the $p$-Laplace operator}

In this section we give spectral estimates of the $p$-Laplace operator on the base Sobolev-Poincar\'e inequalities.

\subsection{The general case}

Recall that a bounded domain $\Omega\subset\mathbb R^n$ is called an $(r,q)$-Sobolev-Poincar\'e domain, $1\leq r,q\leq \infty$, if for any function $f\in L^1_q(\Omega)$, the $(r,q)$-Sobolev-Poincar\'e inequality
$$
\inf\limits_{c\in\mathbb R}\|f-c\mid L_r(\Omega)\|\leq B_{r,q}(\Omega)\|\nabla f\mid L_p(\Omega)\|
$$
holds.

 Weak $(p,q)$-quasiconformal mappings permits us to "transfer" this property from one domain to another.

\begin{thm}
\label{thm:PoincareEnCompPP} 
Let a bounded domain $\Omega\subset\mathbb R^n$ be a $(r,q)$-Sobolev-Poncar\'e domain, $1<q\leq r<\infty$, and there exists a weak $(p,q)$-quasiconformal mapping $\varphi: \Omega\to\widetilde{\Omega}$ of a domain $\Omega$ onto a bounded domain $\widetilde{\Omega}$, possesses the Luzin $N$-property and such that 
$$
M_r(\Omega)=\operatorname{ess}\sup\limits_{x\in \Omega}\left|J(x,\varphi)\right|^{\frac{1}{r}}<\infty.
$$
Then in the domain $\widetilde{\Omega}$ the $(r,p)$-Sobolev-Poincar\'e inequality 
\begin{equation}
\inf\limits_{c\in\mathbb R}\biggl(\int\limits _{\widetilde{\Omega}}|f(x)-c|^{r}~dx \biggr)^{\frac{1}{r}}\leq B_{r,p}(\widetilde{\Omega})\biggl(\int\limits _{\widetilde{\Omega}}
|\nabla f(x)|^{p}~dx\biggr)^{\frac{1}{p}},\,\, f\in W^1_p(\widetilde{\Omega}),
\label{eq:WPPI}
\end{equation} 
holds and 
$$
B_{r,p}(\widetilde{\Omega})\leq K_{p,q}(\Omega)M_r(\Omega) B_{r,q}({\Omega}).
$$
Here $B_{r,q}({\Omega})$ is the best constant in the $(r,q)$-Sobolev-Poincar\'e inequality in the domain $\Omega$.
\end{thm}

\begin{proof}
Let $f\in L^1_p(\widetilde{\Omega})$. By the conditions of the theorem there exists a weak $(p,q)$-quasiconformal homeomorphism 
$\varphi: \Omega\to \widetilde{\Omega}$. Hence, the composition operator 
$$
\varphi^{\ast}: L^1_p(\widetilde{\Omega})\to L^1_q(\Omega)
$$
is bounded. Because $\Omega$ is a bounded $(r,q)$-Sobolev-Poincar\'e domain  $g=\varphi^{\ast}(f)\in W^1_q(\Omega)$.

Using the change of variable formula we obtain: 

\begin{multline}
\inf\limits_{c\in \mathbb R}\biggl(\int\limits_{\widetilde{\Omega}}|f(y)-c|^r~dy\biggr)^{\frac{1}{r}}=\inf\limits_{c\in \mathbb R}
\biggl(\int\limits_{\Omega}|f(\varphi(x))-c|^r|J(x,\varphi)|~dx\biggr)^{\frac{1}{r}}\\
\leq \ess\sup\limits_{x\in \Omega}|J(x,\varphi)|^{\frac{1}{r}}
\inf\limits_{c\in \mathbb R}\biggl(\int\limits_{\Omega}|f(\varphi(x))-c|^r~dx\biggr)^{\frac{1}{r}}
=M_r(\Omega)\inf\limits_{c\in \mathbb R}\biggl(\int\limits_{\Omega}|g(x)-c|^r~dx\biggr)^{\frac{1}{r}}.
\nonumber
\end{multline}
Because the domain $\Omega$ is a $(r,q)$-Sobolev-Poincar\'e domain we have
$$
\inf\limits_{c\in \mathbb R}\biggl(\int\limits_{\Omega}|g(x)-c|^r~dx\biggr)^{\frac{1}{r}}\leq B_{r,q}(\Omega)\biggl(\int\limits_{\Omega}|\nabla g(x)|^q~dx\biggr)^{\frac{1}{q}}.
$$
Hence
$$
\inf\limits_{c\in \mathbb R}\biggl(\int\limits_{\widetilde{\Omega}}|f(y)-c|^r~dy\biggr)^{\frac{1}{r}}\leq
M_r(\Omega)B_{r,q}(\Omega) \|g\mid L^1_q(\Omega)\|.
$$
By Theorem \ref{CompTh} 
$$
\|g\mid L^1_q(\Omega)\| \leq K_{p,q}(\Omega)\|f\mid L^1_p(\widetilde{\Omega})\|.
$$

Therefore
$$
\inf\limits_{c\in \mathbb R}\biggl(\int\limits_{\widetilde{\Omega}}|f(y)-c|^r~dy\biggr)^{\frac{1}{r}}\leq
K_{p,q}(\Omega)M_r(\Omega)B_{r,q}(\Omega)\biggl(\int\limits _{\widetilde{\Omega}}
|\nabla f(x)|^{p}~dx\biggr)^{\frac{1}{p}}.
$$
\end{proof}

The Theorem~\ref{thm:PoincareEnCompPP} immediately implies the following lower estimate for $\mu_{p}(\tilde{\Omega})$:

\begin{thm} 
\label{thm:estimate}
Suppose that there exists a $(p,q)$-quasiconformal homeomorphism  $\varphi: \Omega\to\widetilde{\Omega}$, of a $(r,q)$-Sobolev-Poncar\'e domain $\Omega\subset\mathbb R^n$ onto $\widetilde{\Omega}$, possesses the Luzin $N$-property and such that 
$$
M_p(\Omega)=\operatorname{ess}\sup\limits_{x\in \Omega}\left|J(x,\varphi)\right|^{\frac{1}{p}}<\infty.
$$
Then 
$$
{\mu_{p}(\widetilde{\Omega})}\geq \left(K^p_{p,q}(\Omega) M^p_p(\Omega)B^p_{p,q}(\Omega)\right)^{-1}.
$$
\end{thm}

Boundedness of a Jacobian of a weak $(p,q)$-quasiconformal mapping is a sufficient but restrictive assumption. In this case $\varphi$ is a Lipschitz mapping this condition holds but an image of a Lipschitz domain can not be a domain with external singularities. 

We shall use a more flexible class of weak $(p,q)$-quasiconformal mappings with an integrable Jacobian, which allows us to map Lipschitz domains onto cusp domains. 

\begin{thm}
\label{thm:PoincareEnCompPQ} 
Let a bounded domain $\Omega\subset\mathbb R^n$ be a $(r,q)$-Sobolev-Poncar\'e domain, $1<q\leq r<\infty$, and there exists a weak $(p,q)$-quasiconformal homeomorphism $\varphi: \Omega\to\widetilde{\Omega}$ of a domain $\Omega$ onto a bounded domain $\widetilde{\Omega}$, possesses the Luzin $N$-property and such that
$$ 
M_{r,s}(\Omega)=\biggl(\int\limits _{\Omega}\left|J(x,\varphi)\right|^{\frac{r}{r-s}}~dx\biggl)^{\frac{r-s}{rs}}<\infty
$$
for some $1\leq s<r$.
Then in the domain $\widetilde{\Omega}$ the $(s,p)$-Sobolev-Poincar\'e inequality
\begin{equation}
\biggl(\int\limits _{\widetilde{\Omega}}|f(x)-f_{\widetilde{\Omega}}|^{s}~dx \biggr)^{\frac{1}{s}}\leq B_{s,p}(\widetilde{\Omega})\biggl(\int\limits _{\widetilde{\Omega}}
|\nabla f(x)|^{p}~dx\biggr)^{\frac{1}{p}},\,\, f\in W^1_p(\widetilde{\Omega}),
\label{eq:WPI}
\end{equation} 
holds and 
$$
B_{s,p}(\widetilde{\Omega})\leq K_{p,q}(\Omega)M_{r,s}(\Omega)B_{r,q}(\Omega).
$$
Here $B_{r,q}(\Omega)$ is the best constant in the $(r,q)$-Sobolev-Poincar\'e inequality in the domain $\Omega$.

\end{thm}

\begin{proof}
Let $f\in L^1_p(\widetilde{\Omega})$. By the conditions of the theorem there exists a $(p,q)$-quasiconformal homeomorphism $\varphi: \Omega\to \widetilde{\Omega}$. By Theorem \ref{CompTh} the composition operator 
$$
\varphi^{\ast}: L^1_p(\widetilde{\Omega})\to L^1_q(\Omega)
$$
is bounded. Because the bounded domain $\Omega$ is a $(r,q)$-Sobolev-Poncar\'e domain  $g=\varphi^{\ast}(f)\in W^1_q(\Omega)$.

Using the change of variable formula and the H\"older inequality we obtain: 

\begin{multline}
\inf\limits_{c\in \mathbb R}\biggl(\int\limits_{\widetilde{\Omega}}|f(y)-c|^s~dy\biggr)^{\frac{1}{s}}=\inf\limits_{c\in \mathbb R}
\biggl(\int\limits_{\Omega}|f(\varphi(x))-c|^s|J(x,\varphi)|~dx\biggr)^{\frac{1}{s}}\\
\leq \biggl(\int\limits_{\Omega}|J(x,\varphi)|^{\frac{r}{r-s}}~dx\biggr)^{\frac{r-s}{rs}}
\inf\limits_{c\in \mathbb R}\biggl(\int\limits_{\Omega}|f(\varphi(x))-c|^r~dx\biggr)^{\frac{1}{r}}\\
=M_{r,s}(\Omega)\inf\limits_{c\in \mathbb R}\biggl(\int\limits_{\Omega}|g(x)-c|^r~dx\biggr)^{\frac{1}{r}}.
\nonumber
\end{multline}

 Because the domain $\Omega$ is a $(r,q)$-Sobolev-Poincar\'e domain the following inequality holds:
$$
\inf\limits_{c\in \mathbb R}\biggl(\int\limits_{\Omega}|g(x)-c|^r~dx\biggr)^{\frac{1}{r}}\leq B_{r,q}(\Omega)\biggl(\int\limits_{\Omega}|\nabla g(x)|^q~dx\biggr)^{\frac{1}{q}}.
$$
Combining two previous inequalities we have
$$
\inf\limits_{c\in \mathbb R}\biggl(\int\limits_{\widetilde{\Omega}}|f(y)-c|^s~dy\biggr)^{\frac{1}{s}}\leq
M_{r,s}(\Omega)B_{r,q}(\Omega) \|g\mid L^1_q(\Omega)\|.
$$
By Theorem \ref{CompTh} 
$$
\|g\mid L^1_q(\Omega)\| \leq K_{p,q}(\Omega)\|f\mid L^1_p(\widetilde{\Omega})\|.
$$

Finally we obtain 
$$
\inf\limits_{c\in \mathbb R}\biggl(\int\limits_{\widetilde{\Omega}}|f(y)-c|^s~dy\biggr)^{\frac{1}{s}}\leq 
K_{p,q}(\Omega)M_{r,s}(\Omega) B_{r,q}(\Omega)\biggl(\int\limits_{\widetilde{\Omega}}|\nabla f|^p~dy\biggr)^{\frac{1}{p}}.
$$

It means that 
$$
B_{s,p}(\widetilde{\Omega})\leq K_{p,q}(\Omega)M_{r,s}(\Omega)B_{r,q}(\Omega).
$$

\end{proof}

We are ready to establish the main lower estimate:

\begin{thm}
\label{thm:estrp}
Let a domain $\Omega\subset\mathbb R^n$ be a $(r,q)$-Sobolev-Poncar\'e domain, $1<q<p<r$, and there exists a weak $(p,q)$-quasiconformal homeomorphism $\varphi: \Omega\to\widetilde{\Omega}$ of a domain $\Omega$ onto a bounded domain $\widetilde{\Omega}$, possesses the Luzin $N$-property and such that
$$ 
M_{r,p}(\Omega)=\biggl(\int\limits _{\Omega}\left|J(x,\varphi)\right|^{\frac{r}{r-p}}~dx\biggl)^{\frac{r-p}{rp}}<\infty.
$$
Then  
$$
{\mu_{p}(\widetilde{\Omega})}\geq \left(K^p_{p,q}(\Omega)M^p_{r,p}(\Omega)B^p_{r,q}(\Omega)\right)^{-1}.
$$
\end{thm}

This theorem follows from Theorem~\ref{thm:PoincareEnCompPQ}  and gives the lower estimate of the first non-trivial Neumann eigenvalue of the $p$-Laplace operator in the terms of $(p,q)$-quasiconformal geometry of domains.

\subsection{The anisotropic H\"older singularities}

Define domains $H_g$  with anisotropic H\"older $\gamma$-singularities (introduced in \cite{GGu}):
$$
H_g=\{ x\in\mathbb R^n : 0<x_n<1, 0<x_i<g_i(x_n),
\,i=1,2,\dots,n-1\}.
$$
Here $g_i(\tau)=\tau^{\gamma_i}$, $\gamma_i\geq 1$, $0\leq\tau\leq 1$ are H\"older functions and for the function $G=\prod_{i=1}^{n-1}g_i$ denote by
$$
\gamma=\frac{\log G(\tau)}{\log \tau}+1.
$$
It is evident that $\gamma\geq n$. In the case $g_1=g_2=\dots=g_{n-1}$ we will say that domain $H_g$ is a domain with $\sigma$-H\"older singularity, $\sigma=(\gamma-1)/(n-1)$.
For $g_1(\tau)=g_2(\tau)=\dots=g_{n-1}(\tau)=\tau$ we will use notation $H_1$ instead
of $H_g$.

Define the mapping $\varphi_a: H_1\to H_g$, $a>0$, by
$$
\varphi_a(x)=\left(\frac{x_1}{x_n}g^a_1(x_n),\dots,\frac{x_{n-1}}{x_n}g^a_{n-1}(x_n),x_n^a\right).
$$

\begin{thm}
\label{lemhol}
Let $(n-p)/(\gamma-p)<a<p(n-q)/q(\gamma-p)$. Then the mapping $\varphi_a: H_1\to H_g$, 
be a weak $(p,q)$-quasiconformal mapping, $1<q<p<\gamma$, from the Lipschitz convex domain $H_1$ onto the "cusp"' domain $H_g$ with
$$
K_{p,q}(H_1)\leq a^{-\frac{1}{p}}\sqrt{a^2(\gamma_1^2+...+\gamma_{n-1}^2+1)-2a\sum_{i=1}^{n-1}\gamma_i}\,.
$$
\end{thm}

\begin{proof}
By simple calculations 
$$
\frac{\partial(\varphi_a)_i}{\partial
x_i}=\frac{g^a_i(x_n)}{x_n},\quad
\frac{\partial(\varphi_a)_i}{\partial
x_n}=\frac{-x_ig^a_i(x_n)}{x_n^{2}}+\frac{ax_ig^{a-1}_i(x_n)}{x_n}g'_i(x_n)
\quad\text{and}\quad\frac{\partial(\varphi_a)_n}{\partial
x_n}=ax_n^{a-1}
$$
for any $i=1,...,n-1$. Hence $J(x,\varphi_a)=ax_n^{a-n}G^a(x_n)=ax_n^{a\gamma-n}$, $J(x,\varphi_a)\leq a$ for $a>1$ and
\begin{multline}\label{maps}
D\varphi_a (x)=
\left(\begin{array}{cccc}
x_n^{a\gamma_1-1} & 0 & ... & (a\gamma_1-1)x_1x_n^{a\gamma_1-2}\\
0 & x_n^{a\gamma_2-1} & ... & (a\gamma_2-1)x_2x_n^{a\gamma_2-2}\\
... & ... & ... & ...\\
0 & 0 & ... & ax_n^{a-1}
\end{array} \right)\\
=
x_n^{a-1}\left(\begin{array}{cccc}
x_n^{a\gamma_1-a} & 0 & ... & (a\gamma_1-1)\frac{x_1}{x_n}x_n^{a(\gamma_1-1)}\\
0 & x_n^{a\gamma_2-a} & ... & (a\gamma_2-1)\frac{x_2}{x_n}x_n^{a(\gamma_2-1)}\\
... & ... & ... & ...\\
0 & 0 & ... & a
\end{array} \right).
\end{multline}

Because $0<x_n<1$ and $x_1/x_n<1$ we have the following estimate
\begin{multline}
|D\varphi_a(x)|\leq x_n^{a-1}\sqrt{\sum_{i=1}^{n-1}(a\gamma_i-1)^2+n-1+a^2}\\
= x_n^{a-1}\sqrt{a^2(\gamma_1^2+...+\gamma_{n-1}^2+1)-2a\sum_{i=1}^{n-1}\gamma_i}.
\nonumber
\end{multline}

Then 
\begin{multline*}
K_{p,q}(H_1)=\left(\int\limits_{H_1}\left(\frac{|D\varphi_a(x)|^p}{J(x,\varphi_a)}\right)^{\frac{q}{p-q}}~dx\right)^{\frac{p-q}{pq}}
\\
\leq  \frac{\sqrt{a^2(\gamma_1^2+...+\gamma_{n-1}^2+1)-2a\sum_{i=1}^{n-1}\gamma_i}}{\sqrt[p]{a}}\left(\int\limits_{H_1}x_n^{\frac{(p(a-1)-(a\gamma-n))q}{p-q}}~dx\right)^{\frac{p-q}{pq}}\\
=
\frac{\sqrt{a^2(\gamma_1^2+...+\gamma_{n-1}^2+1)-2a\sum_{i=1}^{n-1}\gamma_i}}{\sqrt[p]{a}} \left(\int\limits_0^1\int\limits_0^{x_n}...\int\limits_0^{x_n}x_n^{\frac{(p(a-1)-(a\gamma-n))q}{p-q}}~dx_1...dx_n\right)^{\frac{p-q}{pq}}\\=
\frac{\sqrt{a^2(\gamma_1^2+...+\gamma_{n-1}^2+1)-2a\sum_{i=1}^{n-1}\gamma_i}}{\sqrt[p]{a}}\left(\int\limits_0^1 x_n^{\frac{(p(a-1)-(a\gamma-n))q}{p-q}+n-1}~dx_n\right)^{\frac{p-q}{pq}}\\
\leq \frac{\sqrt{a^2(\gamma_1^2+...+\gamma_{n-1}^2+1)-2a\sum_{i=1}^{n-1}\gamma_i}}{\sqrt[p]{a}},
\end{multline*}
if $(p+a\gamma -pa)q<np$ or that equivalent $a<p(n-q)/q(\gamma-p)$. 

Now we check that $1<q<np/(p+a\gamma-pa)<p$. The inequality $1<np/(p+a\gamma-pa)$ implies $a<(np-p)/\gamma-p$, but
$$
\frac{np-p}{\gamma-p}<\frac{p(n-q)}{q(\gamma-p)}, \,\,\text{if}\,\,q>1.
$$
The inequality $np/(p+a\gamma-pa)<p$ implies $a>(n-p)/(\gamma-p)$. So, we have that $a\in \left((n-p)/(\gamma-p),p(n-q)/q(\gamma-p)\right)$.
\end{proof}

We are ready to prove spectral estimates in cusp domains.

\begin{thm} 
\label{thm:Holder_Est}
Let 
$$
H_g:=\{ x\in\mathbb R^n : n \geq 3, 0<x_n<1, 0<x_i<x_n^{\gamma_i},
\,i=1,2,\dots,n-1\}
$$  
$\gamma_i \geq 1$, $\gamma:=1+\sum_{i=1}^{n-1}\gamma_i$, $g:=(\gamma_1,...,\gamma_{n-1})$
be domains  with anisotropic H\"older $\gamma$-singularities.

Then for $1<p<\gamma$
$$
\frac{1}{\mu_p(H_g)} \leq \inf\limits_{a\in I_a}
\left(a^2(\gamma_1^2+...+\gamma_{n-1}^2+1)-2a\sum_{i=1}^{n-1}\gamma_i\right)^{\frac{p}{2}}B^p_{r,q}(H_1),
\nonumber
$$
where $I_a=\left(\max\{(n-p)/(\gamma-p),{p(n-q)}/{\gamma q}\},p(n-q)/q(\gamma-p)\right)$ and $B_{r,q}(H_1)$ is the best constant in the $(r,q)$-Sobolev-Poincar\'e inequality in the domain $H_1$, $q\leq r<\frac{nq}{n-q}$.
\end{thm}

\begin{proof}
By Theorem~\ref{lemhol} the mapping $\varphi_a: H_1\to H_g$, $(n-p)/(\gamma-p)<a<p(n-q)/q(\gamma-p)$,
$$
\varphi_a(x)=\left(\frac{x_1}{x_n}g^a_1(x_n),\dots,\frac{x_{n-1}}{x_n}g^a_{n-1}(x_n),x_n^a\right).
$$
maps the convex Lipschitz domain $H_1$ onto the cusp domain $H_g$ and it is a weak $(p,q)$-quasiconformal mapping, $1<q<p<\gamma$.

Let us check conditions of Theorem~\ref{thm:estrp}. Because $\varphi$ is a weak $(p,q)$-quasiconformal mapping then $K_{p,q}(H_1)$ is finite. The basic domain $H_1$  is a Lipschitz domains and so is a $(r,q)$-Sobolev-Poincar\'e domain, i.~e. $B_{r,q}(H_1)<\infty$. 

Now we estimate the constant $M_{r,p}(H_1)$: 

\begin{multline*} 
M_{r,p}(H_1)=\left(\int\limits _{H_1}\left|J(x,\varphi_a)\right|^{\frac{r}{r-p}}~dx\right)^{\frac{r-p}{rp}}= 
a^{\frac{1}{p}}\left(\int\limits _{H_1} \left(x_n^{a\gamma-n}\right)^{\frac{r}{r-p}}~dx\right)^{\frac{r-p}{rp}}\\
=
a^{\frac{1}{p}}\left(\int\limits _{0}^1 \left(x_n^{a\gamma-n}\right)^{\frac{r}{r-p}}\left(\int\limits_0^{x_n}~dx_1~\dots ~ 
\int\limits_0^{x_n}~dx_{n-1}\right)~dx_n\right)^{\frac{r-p}{rp}}\\
=
a^{\frac{1}{p}}\left(\int\limits _{0}^1 \left(x_n^{a\gamma-n}\right)^{\frac{r}{r-p}}\cdot x_n^{n-1}~dx_n\right)^{\frac{r-p}{rp}}<\infty,
\end{multline*}
if 
$$
\frac{(a\gamma-n)r}{r-p}+n-1>-1,\,\,\,\text{i.~e.}\,\,\,a>\frac{np}{\gamma r}.
$$

Because $0<x_n<1$, we have $M_{r,p}(H_1)\leq a^{\frac{1}{p}}$ if $a>np/\gamma r$. If we take $r<nq/(n-q)$ we obtain that 
$$
M_{\frac{nq}{n-q},p}(H_1)\leq a^{\frac{1}{p}},\,\,\text{if}\,\, a>\frac{p(n-q)}{\gamma q}.
$$

The conditions of  Theorem~\ref{thm:estrp} is fulfilled. Therefore\begin{multline}
\frac{1}{\mu_p(H_g)} \leq K^p_{p,q}(H_1)M^p_{\frac{nq}{n-q},p}(H_1)B^p_{\frac{nq}{n-q},q}(H_1)
\\ \leq 
\left(a^2(\gamma_1^2+...+\gamma_{n-1}^2+1)-2a\sum_{i=1}^{n-1}\gamma_i\right)^{\frac{p}{2}}B^p_{r,q}(H_1),
\nonumber
\end{multline}
where $\max\{(n-p)/(\gamma-p),{p(n-q)}/{\gamma q}\}<a<p(n-q)/q(\gamma-p)$ and $B_{r,q}(H_1)$ is the best constant in the $(r,q)$-Sobolev-Poincar\'e inequality in the domain $H_1$ for some  $q\in [r,\frac{nq}{n-q})$.

\end{proof}

Note that in \cite{GU17} we proved the estimate of the Poincar\'e constant in the $(r,q)$-Sobolev-Poincar\'e inequality in the domain $H_1$:

\begin{equation}
\label{H1}
B_{r,q}(H_1)\leq
n\left(\frac{1-\delta}{{1}/{n}-\delta}\right)^{1-\delta}\omega_n^{1-\frac{1}{n}}\left(\frac{1}{(n+1)!}\right)^{\frac{1}{n}-\delta},\,\,\delta=\frac{1}{q}-\frac{1}{r}\geq 0.
\end{equation}

The problem of exact values of constants in the $(r,q)$-Sobolev-Poincar\'e inequalities in the case $p\ne r$ is a complicated open problem even in the case of the unit disc $\mathbb D\subset\mathbb R^2$.

\vskip 0.2cm

\noindent
{\bf Acknowledgments.}
The first author was supported by the United States-Israel Binational Science Foundation (BSF Grant No. 2014055).

\end{document}